\documentclass[a4paper,11pt]{amsart}
\usepackage{amsmath,amssymb,amsthm,amsfonts}
\usepackage{enumerate,graphicx,eucal}
\numberwithin{equation}{section}

\newcommand{\bigpare}[1]{\bigl(#1\bigr)}

\newcommand{\Bigpare}[1]{\Bigl(#1\Bigr)}

\newcommand{\bigbra}[1]{\bigl\{#1\bigr\}}

\newcommand{\bigset}[2]{\bigl\{#1\bigm|#2\bigr\}}

\newcommand{\norm}[1]{\| #1 \|}
\newcommand{\bignorm}[1]{\bigl\| #1 \bigr\|}
\newcommand{\Bignorm}[1]{\Bigl\| #1 \Bigr\|}

\newcommand{\abs}[1]{| #1 |}
\newcommand{\bigabs}[1]{\bigl| #1 \bigr|}

\newcommand{\jap}[1]{\langle #1 \rangle}

\def\a{\alpha}
\def\b{\beta}
\def\c{\gamma}
\def\d{\delta}
\def\e{\varepsilon}
\def\f{\varphi}
\def\g{\psi}
\def\h{\hbar}
\def\k{\kappa}

\def\m{\mu}
\def\n{\nu}
\def\o{\omega}
\def\s{\sigma}
\def\t{\tau}
\def\x{\xi}

\def\th{\theta}

\renewcommand{\O}{\Omega}

\def\re{\mathbb{R}}
\def\co{\mathbb{C}}

\def\pa{\partial}

\renewcommand{\Re}{\text{{\rm Re}\;}}
\renewcommand{\Im}{\text{{\rm Im}\;}}

\newcommand{\supp}{\text{{\rm supp}\;}}

\newcommand{\dist}{\text{\rm dist}}

\newcommand{\res}{\text{\rm res}}
\newcommand{\inter}{\text{\rm int}}
\newcommand{\exter}{\text{\rm ext}}
\newcommand{\wV}{{\widetilde V}}

\DeclareMathOperator*{\slim}{s-lim}

\DeclareMathOperator{\sflow}{sf}
\DeclareMathOperator{\sign}{sign}

\newcommand{\Ran}{\text{\rm Ran\;}}

\newcommand{\Ker}{\mbox{{\rm Ker}}}

\newtheorem{thm}{Theorem}[section]
\newtheorem{lem}[thm]{Lemma}
\newtheorem{prop}[thm]{Proposition}
\newtheorem{cor}[thm]{Corollary}

\theoremstyle{definition}

\newtheorem{ass}{Assumption}

\theoremstyle{remark}
\newtheorem*{rem}{Remark}

\title[scattering matrix near resonant energies]{The spectrum of the scattering matrix near resonant energies in the semiclassical limit}

\author{Shu Nakamura}
\address{
Graduate School of Mathematical Science, 
University of Tokyo, Tokyo, Japan}
\email{shu@ms.u-tokyo.ac.jp}

\author{Alexander Pushnitski}
\address{Department of Mathematics,
King's College London, 
Strand, London, WC2R~2LS, U.K.}
\email{alexander.pushnitski@kcl.ac.uk}

\begin{document}

\begin{abstract}
The object of study in this paper is the on-shell 
scattering matrix $S(E)$ of the Schr\"odinger operator 
with the potential satisfying assumptions typical in the theory of 
shape resonances. We study the spectrum of $S(E)$ in the semiclassical limit
when the energy parameter $E$ varies 
from $E_\text{res}-\varepsilon$ to $E_\text{res}+\varepsilon$,
where $E_\text{res}$ is a real part of a resonance, and $\varepsilon$ 
is sufficiently small. 
The main result of our work describes the spectral flow of the 
scattering matrix through a given point on the unit circle. 
This result is closely related to the Breit-Wigner effect. 
\end{abstract}

\maketitle


\section{Introduction}


\subsection{The set-up}
We consider the Schr\"odinger operator
\begin{equation}
H=H_0+V, \quad H_0=-\h^2\Delta \quad \text{in }L^2(\re^d), \quad d\geq 2,
\label{a0}
\end{equation}
where $\h\in (0,1)$ is the Planck constant and the potential $V=V(x)$ 
satisfies the short-range  condition
\begin{equation}\label{eq:Short-Range-0}
|V(x)|\leq C(1+|x|)^{-\rho}, \quad x\in\re^d,
\end{equation}
with $\rho>1$. 
We will be interested in the semiclassical regime $\h\to+0$, although
the dependence of various operators on $\h$ will be suppressed in our notation. 
For $E>0$ we define the classically accessible region by
\[
\mathcal{G}(E)= \bigset{x\in\re^d}{V(x)< E}
\]
and write
\[
\mathcal{G}(E)= \mathcal{G}^\inter(E)\cup \mathcal{G}^\exter(E), 
\]
where $\mathcal{G}^\exter(E)$ is the unbounded connected component of 
$\mathcal{G}(E)$, and $\mathcal{G}^\inter(E)$ is the union of all other connected 
components. 

In Section~\ref{sec1.2}, we describe our assumptions on $V$;  these are typical for the theory
of shape resonances.
In particular, we require that for some $E_0>0$ the interior domain $\mathcal{G}^\inter(E_0)$ 
is non-empty and that for all energies $E$ in a neighbourhood $\Delta=(E_0-\delta,E_0+\delta)$ of $E_0$ the potential
$V$ is non-trapping in $\mathcal{G}^\exter(E)$; see Assumption~B below and figure~1. 

If it were not for the quantum mechanical tunnelling, the quantum particle with an energy $E\in\Delta$ 
would not be able to penetrate the potential barrier separating $\mathcal{G}^\inter(E)$ from $\mathcal{G}^\exter(E)$. 
Thus, the particle would 
be either confined to the domain $\mathcal{G}^\inter(E)$
or experience scattering in the domain $\mathcal{G}^\exter(E)$.
The particles confined to $\mathcal{G}^\inter(E)$  would  then generate bound states 
with positive energies. Due to tunnelling, these bound states in fact become resonances
with exponentially small (in the semiclassical regime) imaginary part; see e.g. \cite{CDKS,HS,HiS,Na2,GS,Na3,MRS}.
Resonances produced in this way are called \emph{shape resonances}. 

Our purpose is to study the spectrum of the scattering matrix for the pair $H_0,H$ for energies near 
the real parts of shape resonances. In order to locate these resonances,
we use the following standard technique: we define an auxiliary Hamiltonian $H^\inter=H_0+V^\inter$
whose eigenvalues coincide (up to an exponentially small error $O(e^{-c/\h})$) with the real parts of 
shape resonances. 
The potential $V^\inter$ is defined such that $V^\inter(x)=V(x)$  in $\mathcal{G}^\inter(E_0)$ 
and $V^\inter(x)\geq E_+>E_0$ in $\re^d\setminus  \mathcal{G}^\inter(E_0)$; 
the precise assumptions are listed in Section~\ref{sec1.2}, but to get an
at-a-glance idea of our construction, the reader is advised to take a look at 
figure 1. 
We will call the positive eigenvalues of $H^\inter$ the \emph{resonant energies}. 
As mentioned above, under additional assumptions one can prove that 
for each resonant energy $E_\res$ there exists a resonance $\mathcal E_\res$ 
of $H$ with $\abs{E_\res-\Re \mathcal E_\res}$ and $\abs{\Im \mathcal E_\res}$ exponentially
small in the semiclassical regime, see \cite{CDKS,HS,HiS,Na2,Na3,MRS}.
However, it is technically convenient for us to work with real resonant energies $E_\res$
rather than with complex  resonances $\mathcal E_\res$. Thus, although resonances provide motivation 
for our work and are key to interpreting our results, we will say nothing about them 
and instead refer to resonant energies. 
In fact (although this is merely a technical point) 
we do not assume that the resolvent of $H$ admits an analytic continuation 
and so the existence of resonances under our assumptions cannot be guaranteed; 
see \cite{MRS} for a detailed analysis of this issue. 

\begin{rem}
An alternative way to construct $H^\inter$, used e.g. in \cite{CDKS}, is to impose a
Dirichlet boundary condition that decouples the domains $\mathcal{G}^\inter(E_0)$
and $\mathcal{G}^\exter(E_0)$ and then to define $H^\inter$ as the Hamiltonian 
corresponding to the interior domain. 
We find it more convenient to work with the Hamiltonian $H^\inter$ defined on the whole
space. 
\end{rem}

\begin{figure}[h]
\begin{center}
\hspace*{-0.5cm}
\includegraphics[width=16cm]{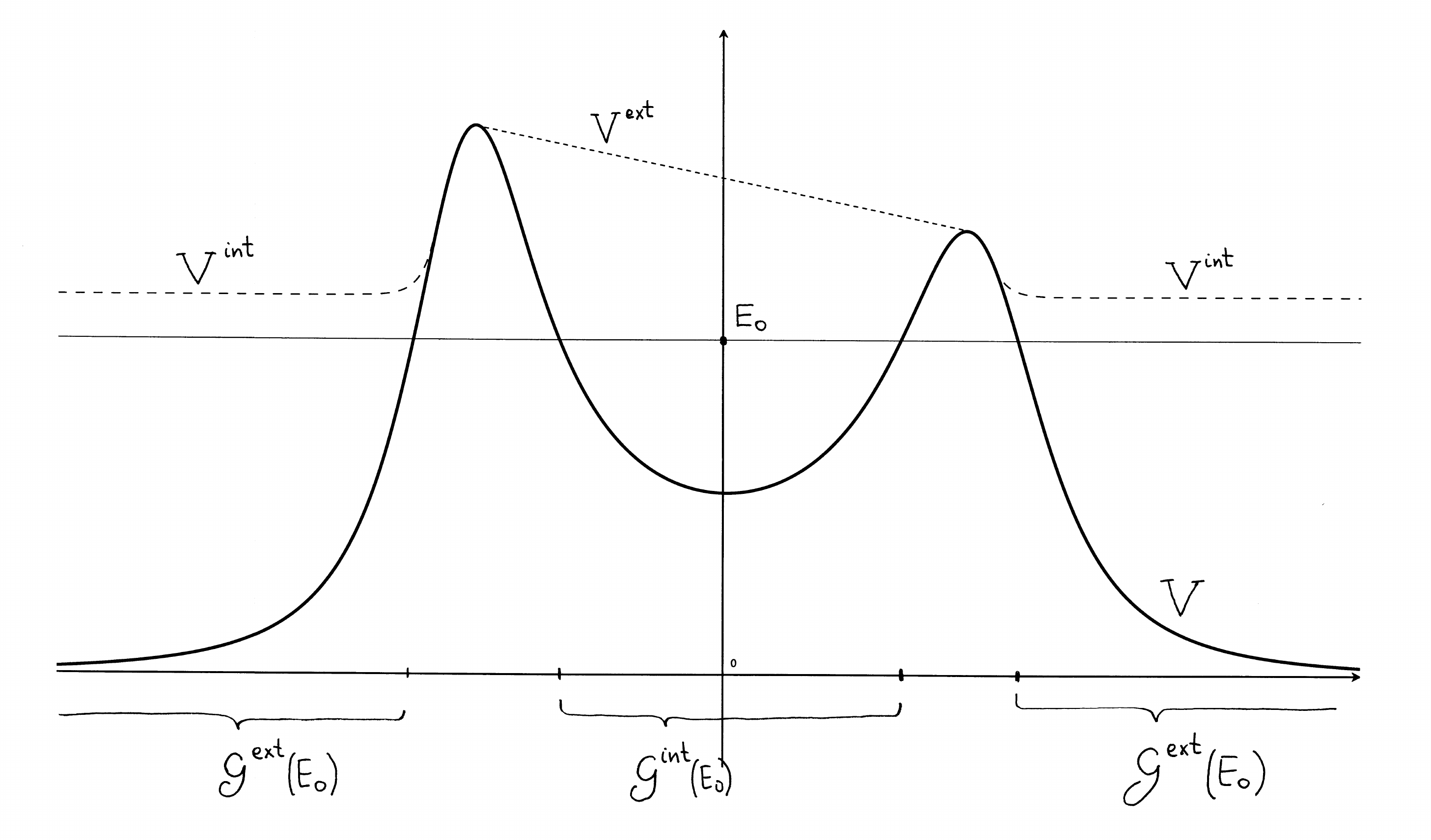}
\end{center}

\caption{Illustrative graph of $V(x)$, $V^\inter(x)$, $V^\exter(x)$}
\end{figure}

Besides $H^\inter$, we also define the Hamiltonian $H^\exter=H_0+V^\exter$, where 
the potential $V^\exter$ coincides with $V$ on $\mathcal{G}^\exter(E_0)$ but is globally
non-trapping, so the domain $\{x\mid V^\exter(x)<E\}$ has no bounded 
connected component; see figure 1. It turns out that away from the 
resonant energies, the scattering matrix for the pair $H,H_0$ is exponentially
close to the scattering matrix for the pair $H^\exter,H_0$, see Proposition~\ref{S-matrix-error-prop}.
We will use the pair $H^\exter,H_0$ as a reference system which has ``almost'' 
the same scattering characteristics as $H,H_0$, but no resonances. 

\subsection{The Breit-Wigner effect}

The main object of this paper is the spectrum of the scattering matrix
$S(E)=S(E;H,H_0)$, where $E>0$ varies near resonant energies.  
We recall the precise definition of the scattering matrix in Section~\ref{sec4.1}. 
The scattering matrix $S(E)$ is a unitary operator 
on $L^2(\mathbb{S}^{d-1})$ and the difference 
$S(E)-I$ is compact, where $I$ is the identity operator. 
Thus, the spectrum of $S(E)$ consists of eigenvalues 
on the unit circle, and the multiplicities of all eigenvalues 
(apart from possibly 1) are finite. 
These eigenvalues may accumulate only to $1$. 
We denote the eigenvalues of $S(E)$, enumerated with multiplicities taken into account, 
by $\bigbra{e^{i\th_n(E)}}_{n=1}^\infty$. 
The scattering matrix $S(E)$ depends continuously on the energy $E>0$ in the operator norm.

If the potential satisfies the short range condition \eqref{eq:Short-Range-0} with $\rho>d$, 
then one can define the spectral shift function $\x(E)=\x(E;H,H_0)$; see, e.g., \cite{BY}
for an introduction to the spectral shift function theory. 
For $E<0$, the spectral shift function coincides with $-N((-\infty,E);H)$;
here and it what follows $N(X,H)$ denotes 
the number of eigenvalues of $H$ in the interval $X$. 
Thus, if $E<0$ is an eigenvalue of $H$ of multiplicity $m$, then 
$\xi$ has a discontinuity at $E$:
$\xi(E+0)-\xi(E-0)=-m$.

For $E>0$, 
the spectral shift function $\x(E)$ is continuous in $E$ and is 
related to the scattering matrix by the Birman-Krein formula 
\begin{equation}
\det S(E)= e^{-2\pi i \x(E)}, \quad E>0.
\label{BK}
\end{equation}
This formula can be equivalently written in terms of the eigenvalues 
$\{e^{i\th_n(E)}\}_{n=1}^\infty$ of $S(E)$ as 
\begin{equation}\label{eq:SSF-EV-1}
\x(E) =-\frac{1}{2\pi} \sum_{n=1}^\infty \th_n(E) \quad (\mbox{mod } 1).
\end{equation}

Suppose $E$ grows monotonically. Then, as $E$ crosses a resonant
energy $E_\res>0$ of multiplicity $m\geq1$, the spectral shift function $\xi(E)$
experiences an exponentially fast (in the semiclassical regime) increment
by $(-m)$:
\begin{prop}[\cite{Na1}]\label{prp.na1}
Let Assumptions~A and B (see Section~\ref{sec1.2}) 
hold true with $\rho>d$ in \eqref{eq:Short-Range-1}, and let $H^\exter=H_0+V^\exter$ 
and $H^\inter=H_0+V^\inter$ be as stipulated in Section~\ref{sec1.2a}.
Then there exist positive constants  $\a$, $\b$, $\d$ such that for all $E$ satisfying 
$|E-E_0|<\d$ and $\dist(E,\s(H^\inter))>e^{-\b/\h}$ one has 
\[
\bigabs{\x(E;H,H_0)-\x(E;H^\exter,H_0)+ N((-\infty,E);H^\inter)} \leq C e^{-\a/\h}.
\]
\end{prop}

Moreover, since $V^\exter$ is non-trapping, 
the behaviour of $\x(E;H^\exter,H_0)$ near $E=E_0$ is well understood, with an asymptotic 
expansion in powers of $\h$ (\cite{RT1, RT2}). Thus, $\x(E;H,H_0)$ can be approximated by a sum 
of the smooth component $\x(E;H^\exter,H_0)$ and a {\em step}\/ component $-N((-\infty,E);H^\inter)$. 
Proposition~\ref{prp.na1} is one of the alternative ways of describing 
the \emph{Breit-Wigner effect}; see \cite[Section~134]{Landau-Lifshits}
or \cite[Chapter~12]{Newton} for a physics discussion or
\cite{GMR,Na2} for precise mathematical results. 
We emphasise that the mathematical description of 
the Breit-Wigner effect requires the trace class assumption $\rho>d$
in \eqref{eq:Short-Range-0},
since the spectral shift function is only defined in the trace 
class framework. On the other hand, the scattering matrix 
is well defined under the assumption $\rho>1$.

Since $\xi(E)$ experiences a fast ``jump'' at resonant energies, 
formula \eqref{eq:SSF-EV-1} suggests that some of the phases $\theta_n(E)$
also experience jumps near $E_\res$. This leads to the following questions:
\begin{enumerate}[(i)]
\item
What is the behaviour of individual phases $\theta_n(E)$ near resonant energies?
\item
Can one observe some version of the Breit-Wigner effect outside the trace class scheme 
by looking at the phases $\theta_n(E)$? 
\end{enumerate}

We attempt to answer these questions, at least partially, in this paper. 
We study the behaviour of the phases $\{\theta_n(E)\}_{n=1}^\infty$
outside the trace class scheme (i.e. under the assumption 
$\rho>1$ in \eqref{eq:Short-Range-0})
when $E$ varies near resonant energies. 
In spectral theory it is often more convenient to study an eigenvalue
counting function instead of individual eigenvalues. 
This turns out to be the case in our problem: instead of looking at 
individual eigenvalues $e^{i\theta_n(E)}$ of the scattering matrix $S(E)$, we study a certain version of 
the eigenvalue counting function, known as the spectral flow. 
Our main result, Theorem~\ref{main-thm}, says, roughly speaking, that
when $E$ increases monotonically from $E_\res-\e$ to $E_\res+\e$, 
where $E_\res$ is a resonant energy 
of multiplicity $m\geq1$ 
and $\e>0$ is exponentially small,  the spectral flow of $S(E)$ through ``most'' points
$e^{i\theta}$ on the unit circle equals $m$. This means that the number
of eigenvalues of $S(E)$ that cross $e^{i\theta}$ anti-clockwise 
minus 
the number 
of eigenvalues of $S(E)$ that cross $e^{i\theta}$ clockwise
equals $m$.

We expect that the main conclusions of our work hold true also in other models where 
resonances are present close to the real axis. The semiclassical set-up for us is 
simply a particular mechanism which produces isolated resonances with a small imaginary part.

\subsection{Acknowledgement}
The authors are grateful to N.~Filonov for a careful critical reading of 
the manuscript and for making a number of useful suggestions.
A.P. is grateful to the Graduate School of Mathematical Science, 
University of Tokyo, 
for hospitality during April 2011.

\section{Main result}\label{sec1a}

\subsection{Assumptions}\label{sec1.2}
Let $H_0$, $H$ be as in \eqref{a0}. 
We will need a version of the short-range condition \eqref{eq:Short-Range-0} 
which involves also the derivatives of $V$: 

\begin{ass}
$V\in C^\infty(\re^d)$ is a real-valued function such that for some $\rho>1$ and any 
multi-index $\a$, 
\begin{equation}\label{eq:Short-Range-1}
\bigabs{\pa_x^\a V(x)}\leq C_\a\jap{x}^{-\rho-|\a|}, \quad x\in\re^d, 
\end{equation}
where $\jap{x}=(1+|x|^2)^{1/2}$. 
\end{ass}

Next, we make a standard non-trapping assumption, 
cf. e.g. \cite{RT1,RT2}.
 Let $(x(t;y,v))_{t\in\re}$ be the solution to the 
Newton equation: 
\begin{align*}
&\ddot x(t;y,v) =-2\nabla V(x(t;y,v)), \\
&x(0;y,v) =y, \quad \dot x(0;y,v)=v.
\end{align*}

\begin{ass}
(i) $\mathcal{G}^\inter(E_0)\neq \varnothing$. \newline
(ii) There exists a neighborhood $\Delta=(E_0-\delta,E_0+\delta)$ of $E_0$ such that all energies $E\in \Delta$ 
are non-trapping in $\mathcal{G}^\exter(E)$ 
in the sense of Robert-Tamura, i.e., for any $R>0$ there is $T>0$ such that if 
\[
y\in \mathcal{G}^\exter(E), \quad |y|<R, \quad \abs{v}^2+V(y)=E, 
\]
then 
\[
|x(t;y,v)|\geq R \quad \text{for }|t|\geq T.
\]
\end{ass}

\subsection{The Hamiltonians $H^\inter$ and $H^\exter$. Resonant energies}\label{sec1.2a}

In order to state our results, we need to introduce two auxiliary Hamiltonians, $H^\inter$ and $H^\exter$. 
Let $\O_1$, $\O_2\subset \re^d$ be open sets such that 
\[
\mathcal{G}^\inter(E_0)\Subset \O_1\Subset \O_2\Subset (\re^d\setminus \mathcal{G}^\exter(E_0)),
\]
where $\O_1\Subset \O_2$ means, as usual, that $\overline{\O}_1\subset \O_2$ 
($\overline{\O}_1$ is the closure of $\O_1$). 
Let us fix $E_+>E_0$ sufficiently close to $E_0$ such that 
\begin{equation}\label{eq:class-forbid-ass}
\inf\bigset{V(x)}{x\in \O_2\setminus \O_1} >E_+.
\end{equation}
Next we choose $V^\inter$ and $V^\exter \in C^\infty(\re^d)$ so that 
(see fig.~1)
\begin{align*}
& V^\exter(x)=V(x), \quad \text{if }x\in\re^d\setminus \O_1, \\
& V^\exter(x)\geq V(x) \quad \text{everywhere}, \\
&V^\exter(x)\geq E_+'> E_+, \quad \text{if }x\in \O_1, \\
&V^\inter(x)=V(x), \quad \text{if }x\in \O_2, \\
&V^\inter(x)\geq E_+'>E_+, \quad\text{if } x\in \re^d\setminus \O_2.
\end{align*}
We assume $V^\inter$ to be bounded; in fact, we may assume $V^\inter$ to be constant
(greater than $E_+$) outside a compact set. We set 
\[
H^\exter=H_0+V^\exter, \quad H^\inter=H_0+V^\inter.
\]
By our assumptions, $H^\inter$ has only discrete spectrum in the interval $(-\infty,E_+)$. 
We will call the eigenvalues of $H^\inter$ in this interval the {\em resonant energies}\/  for $H$. 
Since the above definitions do not uniquely specify $V^\inter$, the resonant energies are not uniquely 
defined. The following statement shows, however, 
that the discrepancy between different definitions of 
resonant energies is exponentially small in the semiclassical limit $\h\to0$.

\begin{prop}\label{prop.res-error}
Let $V^\inter_j$, $j=1,2$, be two choices of the potential $V^\inter$, satisfying the above 
assumptions. 
For $j=1,2$, 
let $E_1^{(j)}\leq E_2^{(j)}\leq\cdots$ 
be the eigenvalues of $H_0+V_j^\inter$ in the interval $(-\infty,E_+)$, 
listed in non-decreasing order with multiplicities taken into account. 
Then there exists $\n>0$ such that for all $n$ and for all sufficiently 
small $\h>0$, the estimate 
\begin{equation}
\abs{E_n^{(1)}-E_n^{(2)}}\leq e^{-\n/\h}
\label{eq:exp}
\end{equation}
holds true. 
\end{prop}
For the proof, see Appendix~\ref{app.c}.
We note that a lower bound for the  constant $\n>0$ in the estimate
\eqref{eq:exp} is explicitly given 
by the Agmon distance between $\O_1$ 
and $\O_2^c$ at the energy $E_+$.

In the problem we are discussing one has to keep in mind two scales as $\h\to0$:
the power scale and the exponential scale. Indeed, the number
of eigenvalues of $H^\inter$ on $(-\infty,E_+)$ grows as $\h^{-d}$. 
On the other hand, our results below are
valid for energies in $\Delta$ outside exponentially small neighbourhoods
of resonant energies.
More precisely, we will consider the energies
$E\in\Delta$ which satisfy $\dist(E,\s(H^\inter))> e^{-\alpha/\h}$ for some
$\alpha>0$. 
Thus, the Lebesgue measure of the set
\[
\bigset{E\in\Delta}{\dist(E,\s(H^\inter))\leq e^{-\alpha/\h}}
\]
that we exclude from the interval $\Delta$  
is exponentially small as $\h\to 0$. 
The fact that we have to exclude exponentially small neighbourhoods 
of resonant energies $E_\res$ is a reflection of the effect that the 
resonances $\mathcal E_\res$ of $H$ are exponentially close to $E_\res$,
see e.g. \cite{HS,CDKS,Na2}.

\subsection{The scattering matrix}
The following preliminary result (which is not really new, cf. \cite{Na2}) shows that away from the resonant energies 
the scattering matrices $S(E;H,H_0)$ and $S(E;H^\exter, H_0)$ are exponentially close to each other: 

\begin{prop}\label{S-matrix-error-prop}
There exist positive constants $C$, $\a$, $\d$ such that for all $E$ satisfying $|E-E_0|<\d$ and 
$\dist (E,\s(H^\inter))> e^{-\alpha/\h}$, one has 
\begin{equation}\label{eq:Smatrix-difference-1}
\norm{S(E;H^\exter,H_0)-S(E;H,H_0)}\leq C e^{-\a/\h},
\end{equation}
for all sufficiently small $\h>0$. 
\end{prop}
The proof follows directly from Lemma~\ref{non-reso-S-mat-lem}
and the representation \eqref{eq:Smatrix-chain-rule} below. 
Proposition~\ref{S-matrix-error-prop}, in particular, immediately implies that 
away from resonant energies the scattering matrix 
$S(E;H^\exter,H_0)$ is independent of the choice of $V^\exter$ up to an 
exponentially small error.

The next preliminary result shows that $S(E;H^\exter,H_0)$ varies sufficiently slowly:
\begin{prop}\label{Smatrix-Hoelder}
There exist positive constants $C$, $\c$, $\d$ such that if $E_1,E_2\in (E_0-\d,E_0+\d)$ then 
\begin{equation}\label{eq:Smatrix-Hoelder}
\norm{S(E_1;H^\exter,H_0)-S(E_2;H^\exter,H_0)} \leq C \h^{-2-\c} |E_1-E_2|^\c, 
\quad \h\in (0,1].
\end{equation}
\end{prop}
See Appendix~\ref{Smatrix-Hoelder-proof} for the proof. 
Note that the neighbourhoods of resonant energies are not excluded
in Proposition~\ref{Smatrix-Hoelder}.

\subsection{Main result}

In order to state our main result, 
first we need to recall the definition of spectral flow for 
unitary operators. 
Let $\{U(t)\}_{t\in[0,1]}$ be a norm continuous family of unitary operators 
in a Hilbert space $\mathcal{H}$ such that $U(t)-I$ is compact for all $t\in [0,1]$. 
We would like to define the spectral flow of $\{U(t)\}_{t\in[0,1]}$ through a point $e^{i\th}$, 
$\th\in (0,2\pi)$. The {\em na\"ive}\/ definition of spectral flow is 
\begin{equation}\label{eq:naive-sf-definition}
\begin{split}
&\mbox{sf}(e^{i\th};\{U(t)\}_{t\in[0,1]}) \\
&= \jap{\text{the number of eigenvalues of $U(t)$ which cross $e^{i\th}$} \\
&\qquad \text{in the anti-clockwise direction}}\\
&\quad - \jap{\text{the number of eigenvalues of $U(t)$ which cross $e^{i\th}$} \\
&\quad\qquad  \text{in the clockwise direction}}
\end{split}
\end{equation}
as $t$ grows from 0 to 1. The eigenvalues are counted with multiplicities taken into account. 
Of course, there may be infinitely many intersections, and so in general the r.h.s.\ of 
\eqref{eq:naive-sf-definition} may be ill-defined. We postpone the 
discussion of the precise definition of the spectral flow until Section~\ref{sec2}.

Our main result is 
\begin{thm}\label{main-thm}
Under Assumptions A and B (see Section~\ref{sec1.2})
there exist positive constants $\alpha$, $\beta$, $\delta$ such that
the following statement holds true. 
Suppose that $\abs{E_\res-E_0}<\delta$, $E_\res$ is an eigenvalue of $H^\inter$ 
of multiplicity $m\geq1$ and 
$$
[E_\res-e^{-\alpha/\h},E_\res+e^{-\alpha/\h}]\cap \sigma(H^\inter)=\{E_\res \}.
$$
Then for all $\theta\in(0,2\pi)$ such that
\begin{equation}
\dist\bigl(e^{i\theta},\sigma(S(E_\res;H^\exter,H_0))\bigr)>e^{-\beta/\h}
\label{eq:theta}
\end{equation}
and for all sufficiently small $\h>0$ one has
\begin{equation}
\sflow (e^{i\theta};\{S(E;H,H_0)\}_{E\in[E_-,E_+]})=m,
\label{a1}
\end{equation}
where 
$E_\pm=E_\res\pm e^{-\alpha/\h}$.
\end{thm}
It would be interesting to obtain more detailed information about the
eigenvalues of the scattering matrix near resonant energies. 

Theorem~\ref{main-thm} will be derived in Sections~\ref{sec2}--\ref{sec4} from a more precise
result, Theorem~\ref{main-theorem}, 
which is stated in terms of the spectral flow of the scattering matrix 
when the energy $E$ varies from some finite value to infinity.

We need to explain that Theorem~\ref{main-thm}  is not vacuous, 
i.e. that the set of angles $\theta$ satisfying \eqref{eq:theta} is 
non-empty:
\begin{prop}\label{S-mat-ev-bound}
There exist positive constants $\d$, $\eta$ such that for any $\th_1,\th_2\in (0,2\pi)$ and 
any $E$ satisfying $|E-E_0|<\d$, one has
\[
\sum_{\theta_1<\theta<\theta_2}
\dim\Ker(S(E;H^\exter,H_0)-e^{i\theta})
= O(\h^{-\eta}), \quad \h\to 0.
\]
\end{prop}
The proof is given in Appendix~\ref{S-mat-ev-bound-proof}.
Proposition~\ref{S-mat-ev-bound} shows that the Lebesgue measure of 
the $O(e^{-\a/\h})$-neighborhood of the spectrum 
of $S(E;H^\exter,H_0)$ on any arc not containing 1 is exponentially small. 
Thus \eqref{eq:theta} holds true for a {\em large}\/ 
set of $\th$.

\subsection{Method of proof}

We will use the chain rule for scattering matrices to write 
\begin{equation}\label{eq:Smatrix-chain-rule}
S(E;H,H_0)=\widetilde S(E;H,H^\exter) S(E;H^\exter,H_0), \quad E>0,
\end{equation}
where $\widetilde S(E;H,H^\exter)$ is 
unitarily equivalent to $S(E;H,H^\exter)$; 
see \eqref{stilde1}, \eqref{stilde2}.
In Section~\ref{sec3} we develop a version of perturbation theory for the 
spectral flow of products of unitaries. This allows us to estimate
the spectral flow of $S(E;H,H_0)$ in terms of the spectral flows of 
$S(E;H,H^\exter)$ and $S(E;H^\exter,H_0)$. 
Next, using the stationary representation for the scattering matrix
(see Section~\ref{sec3.2}) and tunnelling estimates (see Section~\ref{sec4}), 
we show that if $\theta$ satisfies \eqref{eq:theta}, then 
the spectral flow of $S(E;H^\exter,H_0)$ through $e^{i\theta}$ 
is zero and the spectral flow of $S(E;H,H^\exter)$ is $m$.


\section{The eigenvalue counting function for the scattering matrix}\label{sec2}

\subsection{Scattering theory}\label{sec4.1}

Here we recall a small amount of general scattering theory, as necessary to define the basic objects of our 
construction. 
For the details, see e.g. \cite{Yafaev}.
For a self-adjoint operator $A$ we denote by $P_{ac}(A)$ the projection onto the absolutely 
continuous subspace of $A$. 
The wave operators are defined, as usual, by 
\[
W_\pm(B,A) =\slim_{t\to\pm\infty} e^{itB} e^{-itA}P_{ac}(A),
\]
provided that the strong limits exist. The scattering operator is defined by 
\begin{equation}\label{eq:S-Op-Def}
\mathbf S(B,A)= W_+(B,A)^* W_-(B,A).
\end{equation}
If the wave operators are complete, i.e., 
\[
\Ran W_+(B,A)=\Ran W_-(B,A)=\Ran P_{ac}(B),
\]
then the scattering operator $\mathbf S(B,A)$ is unitary in $\Ran P_{ac}(A)$. 
We note the chain rule 
\begin{equation}\label{eq:S-op-chain-rule}
\mathbf S(C,A)=\widetilde{\mathbf S}(C,B)\mathbf S(B,A),
\end{equation}
where 
\begin{equation}\label{eq:def-tilde-S-op}
\widetilde{\mathbf S}(C,B)=W_+(B,A)^* \mathbf S(C,B) W_+(B,A),
\end{equation}
provided that the wave operators $W_\pm(B,A)$ and $W_\pm(C,B)$ exist and are complete. 

By the intertwining property of wave operators, viz.
\[
W_\pm(B,A)A=BW_\pm(B,A),
\]
the scattering operator $\mathbf S(B,A)$ commutes with $A$, 
and therefore $\mathbf S(B,A)$ and $A$ can be simultaneously diagonalized. 
The fibre operators of 
$\mathbf S(B,A)$ in this diagonalisation give the scattering matrix $S(E;B,A)$. 

Now let us recall the details of this construction for the operators
$A=H_0=-\h^2\Delta$, $B=H^\exter$, $C=H$ as defined in Section~\ref{sec1.2a}. 
We first note that the wave operators $W_\pm(H,H_0)$, $W_\pm(H^\exter,H_0)$ and $W_\pm(H,H^\exter)$ 
exist and are complete. 
For $f\in L^2(\re^d)$ such that $\jap{x}^s f(x)\in L^2(\re^d)$, $s>1/2$, we set 
\begin{equation}\label{eq:Free-Sp-Rep}
(\mathcal{F}_E f)(\o) =\frac{1}{\sqrt{2}\h^{d/2}} E^{(d-2)/4}\hat f(\sqrt{E}\o\h^{-1}), 
\quad E>0,\quad \o\in\mathbb{S}^{d-1},
\end{equation}
where $\hat f$ is the (unitary) Fourier transform of $f$. Then the map 
\[
\mathcal{F}\ :\ L^2(\re^d)\to L^2((0,\infty);L^2(\mathbb{S}^{d-1})), \quad 
(\mathcal{F}f)(E,\o)= (\mathcal{F}_E f)(\o),
\]
is a unitary operator which diagonalises $H_0$: 
\[
(\mathcal{F}_E H_0 f)(\o)=E (\mathcal{F}_E f)(\o).
\]
It follows that $\mathcal{F} S(H,H_0)\mathcal{F}^*$ can be represented as a direct integral of fibre operators: 
\[
\mathcal{F}_E \mathbf S(H,H_0)f = S(E;H,H_0)(\mathcal{F}_E f), \quad E>0. 
\]
Here $S(E;H,H_0)$ : $L^2(\mathbb{S}^{d-1}) \to L^2(\mathbb{S}^{d-1})$ is the scattering matrix. 

In the same way, one defines the scattering matrix $S(E;H^\exter,H_0)$. Finally, as in \eqref{eq:S-op-chain-rule}, 
\eqref{eq:def-tilde-S-op}, we denote
\begin{equation}
\widetilde{\mathbf S}(H,H^\exter)= W_+(H^\exter,H_0)^* \mathbf S(H,H^\exter) W_+(H^\exter,H_0).
\label{stilde1}
\end{equation}
Then $\widetilde{\mathbf  S}(H,H^\exter)$ commutes with $H_0$ and therefore 
$\widetilde{\mathbf S}(H,H^\exter)$ can also be represented as 
\begin{equation}
\mathcal{F}_E \widetilde{\mathbf S}(H,H^\exter) f 
=
\widetilde S(E;H,H^\exter)(\mathcal{F}_E f), \quad E>0,
\label{stilde2}
\end{equation}
with some fibre operators $\widetilde S(E;H,H^\exter)$ : $L^2(\mathbb{S}^{d-1})\to L^2(\mathbb{S}^{d-1})$. 
Then, as a consequence of \eqref{eq:S-op-chain-rule}, \eqref{eq:def-tilde-S-op}, the chain rule 
\eqref{eq:Smatrix-chain-rule} holds. 

\subsection{The stationary representation for the scattering matrix}
\label{sec3.2}

Essential for our construction is the stationary representation for the scattering 
matrix $S(E;H,H_0)$, see e.g. \cite[Section~0.7]{Yafaev}.
Let $\mathcal{F}_E$ be as in \eqref{eq:Free-Sp-Rep}, and let 
$R_0(z)=(H_0-zI)^{-1}$, $R(z)=(H-zI)^{-1}$. The stationary representation reads 
\begin{equation}\label{eq:S-mat-rep-1}
S(E;H,H_0)=I-2\pi i \mathcal{F}_E (V-VR(E+i0)V)\mathcal{F}_E^*.
\end{equation}
This representation allows one to describe the spectrum of the scattering matrix in terms of the 
spectrum of the boundary values of the resolvent $R(E+i0)$. 
Let us denote 
\begin{align*}
&J=\sign(V),  \qquad T_0(z)= \sqrt{\abs{V}} R_0(z)\sqrt{\abs{V}},
\\
& A_0(E)=\Re T_0(E+i0), \quad B_0(E)=\Im T_0(E+i0).
\end{align*}
Using this notation, the resolvent identity and the identity $J^{-1}=J$, 
one can write \eqref{eq:S-mat-rep-1} as
\begin{equation}\label{eq:S-mat-rep-1a}
S(E;H,H_0)
=I-2\pi i \mathcal{F}_E \sqrt{\abs{V}}\bigl(J-A_0(E)-iB_0(E)\bigr)^{-1}\sqrt{\abs{V}}\mathcal{F}_E^*;
\end{equation}
see e.g. \cite[Section 7.7]{Yafaev}.
This formula has a general operator theoretic nature and is not specific to the pair $H_0,H$.
In fact, in what follows 
we will apply the representation \eqref{eq:S-mat-rep-1a} 
to the pair $H,H^\exter$.

\subsection{Definition of spectral flow}
First we fix the notation for an eigenvalue counting function of a unitary operator. 
Let $U$ be a unitary operator such that the difference $U-I$ is compact. 
For $\th_1,\th_2\in (0,2\pi)$, we denote 
\begin{equation}\label{eq:EV-Count-1}
N(e^{i\th_1},e^{i\th_2};U)=\sum_{\th_1\leq \th<\th_2}\mbox{dim Ker}(U-e^{i\th}I), 
\quad \text{if }\th_1\leq \th_2, 
\end{equation}
and 
\begin{equation}\label{eq:EV-Count-2}
N(e^{i\th_1},e^{i\th_2};U)=-N(e^{i\th_2},e^{i\th_1};U),  
\quad \text{if }\th_1\geq \th_2. 
\end{equation}

Let $\{U(t)\}_{t\in[0,1]}$ be a norm continuous family of unitary operators 
in a Hilbert space $\mathcal{H}$ such that $U(t)-I$ is compact for all $t\in [0,1]$. 
The \emph{na\"ive} definition of  the spectral flow of $\{U(t)\}_{t\in[0,1]}$ through a point $e^{i\th}$, 
$\th\in (0,2\pi)$, is given by \eqref{eq:naive-sf-definition}. 
Let us discuss a rigorous definition of spectral flow. 
First we assume that there exists $\th_0\in (0,2\pi)$ such that 
\begin{equation}\label{eq:good-theta}
e^{i\th_0}\notin \s(U(t)), \quad \text{for all }t\in [0,1].
\end{equation}
Then we set (using the notation \eqref{eq:EV-Count-1}, \eqref{eq:EV-Count-2})
\begin{equation}\label{eq:sf-def-1}
\mbox{sf}(e^{i\th};\{U(t)\}_{t\in[0,1]}) = N(e^{i\th},e^{i\th_0}; U(1))
-N(e^{i\th},e^{i\th_0};U(0)). 
\end{equation}
It is evident that this definition is independent of the choice of $\th_0$ and agrees with the 
{\em na\"ive}\/ definition \eqref{eq:naive-sf-definition} whenever the latter makes sense. 

In general, $\th_0$ as above may not exist. However, using the norm continuity of $U(t)$, 
one can always find values 
\[
0=t_0<t_1<t_2<\cdots <t_n =1
\]
such that for each of the intervals $\Delta_j =[t_{j-1},t_j]$, a point $\th_j\in (0,2\pi)$ satisfying 
\eqref{eq:good-theta} for all $t\in \Delta_j$ can be found. Thus, the spectral flows 
$\mbox{sf}(\cdot;\{U(t)\}_{t\in \Delta_j})$, $j=1,\dots,n$, are well defined. Now we set 
\begin{equation}\label{eq:sf-def-2}
\mbox{sf}(e^{i\th};\{U(t)\}_{t\in[0,1]}) = \sum_{j=1}^n \mbox{sf}(e^{i\th};\{U(t)\}_{t\in \Delta_j}).
\end{equation}
It is not difficult to see that this definition is independent of the choice of the intervals $\Delta_j$
and the corresponding points $\th_0$, and agrees with the \emph{na\"ive} definition 
\eqref{eq:naive-sf-definition} whenever the latter makes sense. 

In the context of self-adjoint operators with discrete spectrum, the notion of spectral flow goes back 
at least to the seminal work \cite{APS}; see also \cite{RoSa} for a comprehensive survey. 
One can find many equivalent approaches to the definition of spectral flow in the literature. 

The property of the spectral flow that is crucial for us in the sequel is its invariance with respect 
to the homotopies of the family $\{U(t)\}_{t\in[0,1]}$. The homotopy must be of a class preserving the 
compactness of $U(t)-I$. The homotopy invariance of spectral flow is proven by using standard 
topological arguments.

\subsection{The eigenvalue counting function for $S(E)$}\label{sec1.5}

For $E>0$, let us consider the eigenvalue counting function of the scattering matrix 
$S(E)=S(E;H,H_0)$. 
For a fixed $\th_0\in (0,2\pi)$, the function 
\begin{equation}\label{eq:EV-Count-3}
(0,2\pi)\ni \th \mapsto N(e^{i\th},e^{i\th_0};S(E))
\end{equation}
is a non-increasing integer-valued function with jumps at the points $e^{i\th}\in\s(S(E))$. 
Of course, different choices of $\th_0$ lead to different integer additive constants in the 
definition of the counting function \eqref{eq:EV-Count-3}. 
Below we explain how to fix this constant 
in a certain standard way; the resulting counting function will be denoted by 
$\m(e^{i\th},E)$.

Recall the well known relation 
\begin{equation}\label{eq:HighEnergyS-matrix}
\norm{S(E)-I}\to 0 \quad \text{as }E\to+\infty.
\end{equation}
Now fix $E>0$
and define the family
\begin{equation}
U(t)= S(E'), \quad E'= E+\frac{1-t}{t}, \quad t\in (0,1], 
\quad \text{ and } U(0)=I.
\label{familyU}
\end{equation}
Then $\{U(t)\}_{t\in[0,1]}$  is a norm continuous unitary family with $U(t)-I$ compact for all $t\in [0,1]$. 
We set
\begin{equation}\label{eq:mu-and-sf}
\m(e^{i\th},E;H,H_0)
= 
\sflow(e^{i\th},\{U(t)\}_{t\in[0,1]}), 
\quad \text{ where $U$ is given by \eqref{familyU}.}
\end{equation}
Then $\m(\cdot,E;H,H_0)$ coincides with the eigenvalue counting function 
\eqref{eq:EV-Count-3} up to a particular choice of the additive integer normalisation
constant.

The above definition of $\mu$ can be alternatively described as follows. 
The function $\m(e^{i\th},E)$ is the unique function
which satisfies the following properties: 
\begin{enumerate}
\renewcommand{\theenumi}{\roman{enumi}}
\renewcommand{\labelenumi}{(\theenumi) }
\item 
for any $\th_0\in (0,2\pi)$ and any $E>0$ there is an integer $m(\th_0,E)$ such that 
\[
\m(e^{i\th},E)= N(e^{i\th},e^{i\th_0};S(E)) +m(\th_0,E), 
\quad \forall \theta\in(0,2\pi);
\]
\item 
the function
\begin{equation}\label{eq:EV-Count-5}
(0,2\pi)\ni \th \mapsto \m(e^{i\th},E)
\end{equation}
considered as an element of $L^1_{\mbox{\tiny loc}}(0,2\pi)$ (the choice of 
the function space is not important here) depends continuously on $E>0$;

\item 
the function \eqref{eq:EV-Count-5}, considered as an element of 
$L^1_{\mbox{\tiny loc}}(0,2\pi)$, converges to zero as $E\to+\infty$. 
\end{enumerate}

The function $\m$ has been systematically studied in \cite{Push1, Push2} in an abstract 
operator theoretic framework. It possesses a number of natural properties; for example, 
$\pm \m(e^{i\th}, E)\geq 0$ if $\mp V\geq 0$.

\subsection{The counting function $\mu$ near resonant values}\label{sec1.6}

Let the Hamiltonians $H_0, H, H^\inter,H^\exter$, and the energy $E_0$ be as stipulated in Sections~\ref{sec1.2},
\ref{sec1.2a}.
Our results below involve the eigenvalue counting functions $\m(e^{i\th}, E;H,H_0)$, 
$\m(e^{i\th}, E;H^\exter,H_0)$ of the scattering matrices $S(E;H,H_0)$, $S(E;H^\exter,H_0)$; 
see definition \eqref{eq:mu-and-sf} above.

\begin{thm}\label{main-theorem}
There exist positive constants $\a$,  $\d$ such that for all $E$ satisfying $|E-E_0|<\d$ and 
$\dist(E;\s(H^\inter))>\frac12 e^{-\alpha/\h}$, and for all $\th\in (e^{-\a/\h},2\pi-e^{-\a/\h})$, one has 
\begin{multline}\label{eq:main-result-ineq}
\m(e^{i\th_+},E;H^\exter,H_0) +N((-\infty,E);H^\inter) 
\leq \m(e^{i\th},E;H,H_0)  
\\
\leq \m(e^{i\th_-},E;H^\exter,H_0) + N((-\infty,E); H^\inter)
\end{multline}
if $\h>0$ is sufficiently small, where $\th_\pm =\th\pm e^{-\a/\h}$. 
\end{thm}

We note that the factor $\frac12$ in front of the exponential $e^{-\alpha/\h}$ in the statement 
of the theorem is of no importance; it is introduced purely for the convenience 
of the proof of Theorem~\ref{main-thm}.
The proof of Theorem~\ref{main-theorem} is given in Sections~\ref{sec3} and \ref{sec4}. 
\begin{cor}\label{cor}
There exist positive constants $\a$, $\d$ such that for all $E$ satisfying $|E-E_0|<\d$ and 
$\dist(E;\s(H^\inter))>\frac12 e^{-\alpha/\h}$, and for any $\th\in (0,2\pi)$ such that
\begin{equation}\label{eq:small-arc}
\dist(e^{i\theta},\sigma(S(E;H^\exter,H_0)))>e^{-\alpha/\h},
\end{equation}
one has 
\begin{equation}\label{1}
\m(e^{i\th},E;H, H_0) =\m(e^{i\th},E;H^\exter,H_0)+N((-\infty,E);H^\inter). 
\end{equation}
\end{cor}
\begin{proof}
Indeed, for all $\theta$ that satisfy \eqref{eq:small-arc}
we have
\[
\m(e^{i\th_+},E;H^\exter,H_0) 
=
\m(e^{i\th_-},E;H^\exter,H_0) 
=
\m(e^{i\th},E;H^\exter,H_0),
\]
and so \eqref{eq:main-result-ineq} yields \eqref{1}.
\end{proof}
Now using  Corollary~\ref{cor}, 
we can prove Theorem~\ref{main-thm}. 

\begin{proof}[Proof of Theorem~\ref{main-thm}]
Let $\alpha$, $\delta$ be as in Theorem~\ref{main-theorem}.
Using Proposition~\ref{Smatrix-Hoelder}, we get for all sufficiently 
small $\h$
\begin{equation}
\norm{S(E;H^\exter,H_0)-S(E_\res;H^\exter,H_0)}
\leq
C\h^{-2-\gamma}e^{-\gamma\alpha/\h},
\quad\text{ if } \abs{E-E_\res}<e^{-\alpha/\h}.
\label{b1}
\end{equation}
Choose $\beta>0$ such that $\beta<\alpha$ and $\beta<\gamma\alpha$, 
and let $\theta$ be any angle that satisfies \eqref{eq:theta}. Then, combining
\eqref{eq:theta} and \eqref{b1}, we obtain
\begin{equation}
\dist(e^{i\theta},\sigma(S(E;H^\exter,H_0)))>\frac12 e^{-\alpha/\h}
\label{b2}
\end{equation}
for all sufficiently small $\h$ and all $E\in[E_-,E_+]$. 
Now we can apply Corollary~\ref{cor}
to $E=E_\pm$. This yields
\begin{multline}
\mu(e^{i\theta},E_+;H,H_0)-\mu(e^{i\theta},E_-;H,H_0)
\\
=
N((-\infty,E_+);H^\inter)-N((-\infty,E_-);H^\inter)=m.
\label{b3}
\end{multline}
Here we have used the fact that by \eqref{b2}, 
none of the eigenvalues of $S(E;H^\exter,H_0)$ crosses
$e^{i\theta}$ as $E$ grows from $E_-$ to $E_+$, and therefore
$$
\mu(e^{i\theta},E_+;H^\exter,H_0)=\mu(e^{i\theta},E_-;H^\exter,H_0).
$$
Representing the family $\{S(E';H,H_0)\}_{E'\in[E_-,\infty)}$
as the concatenation of the families
$\{S(E';H,H_0)\}_{E'\in[E_-,E_+]}$ and $\{S(E';H,H_0)\}_{E'\in[E_+,\infty)}$
and recalling the definition \eqref{eq:mu-and-sf} of $\mu$, we obtain
\begin{equation}
\mu(e^{i\theta},E_+;H,H_0)-\mu(e^{i\theta},E_-;H,H_0)
=
\sflow (e^{i\theta};\{S(E;H,H_0)\}_{E\in[E_-,E_+]}).
\label{b4}
\end{equation}
From \eqref{b3} and \eqref{b4} we obtain
 \eqref{a1}.  
\end{proof}

\subsection{The strategy of proof of Theorem~\ref{main-theorem}}
The proof is based on the chain rule \eqref{eq:Smatrix-chain-rule}.
In Section~\ref{sec3} we develop a perturbation theory for spectral flow 
of products of unitaries. This allows us to estimate the function
$\mu(\cdot,E;H,H_0)$ in terms of $\mu(\cdot,E;H^\exter,H_0)$ 
and $\mu(\cdot,E;H,H^\exter)$, see Theorem~\ref{sf-ptb-thm}.
The next crucial step is to prove the equality
\begin{equation}
\mu(e^{i\theta},E;H,H^\exter)=N((-\infty,E);H^\inter)
\label{b5}
\end{equation}
for relevant values of $\theta$. 
For this, we rely 
on an explicit  formula for the function $\mu$ 
from \cite{Push1}.
This formula has an abstract operator-theoretic nature; 
we apply it to the pair of operators 
$(H,H^\exter)$. 
Denote $V_0= V^\exter-V$; by our assumptions, we have $V_0\geq0$.  
Similarly to the notation of Section~\ref{sec3.2}, let us set
\begin{align*}
& R^\exter(z)=(H^\exter-zI)^{-1}, \\
& A^\exter(E)=\Re \bigpare{\sqrt{V_0} R^\exter(E+i0)\sqrt{V_0}}, \\
& B^\exter(E)=\Im \bigpare{\sqrt{V_0} R^\exter(E+i0)\sqrt{V_0}}.
\end{align*}
\begin{prop}\label{BS-count-prop}\cite[Section~5]{Push1}
For any $\th\in (0,2\pi)$ and $E>0$, one has 
\begin{align}
&\dim\Ker (S(E;H,H^\exter)-e^{i\th}I)  \nonumber \\
&\qquad =\dim\Ker (A^\exter(E)+\cot (\th/2) B^\exter(E)-I), \label{eq:S-mat-kernel}\\
&\m(e^{i\th},E;H,H^\exter) =N((1,\infty);A^\exter(E)+\cot(\th/2)B^\exter(E)).\label{eq:BS-count-formula}
\end{align}
\end{prop}
See also \cite[Section~4]{Push2} for an alternative proof. 
Proposition~\ref{BS-count-prop} is a consequence of the stationary 
representation \eqref{eq:S-mat-rep-1a} for the scattering matrix. 
This circle of ideas goes back to 
\cite{SY} and perhaps even to \cite{Kato}.

Now we can prove \eqref{b5} as follows. 
According to \eqref{eq:BS-count-formula} with $\theta=\pi$, 
\begin{equation}\label{eq:counting-eq-lem-1}
\m(-1,E;H,H^\exter) =N((1,\infty),A^\exter(E)).
\end{equation}
On the other hand, by the Birman-Schwinger principle, 
\begin{equation}\label{eq:counting-eq-lem-2}
N((-\infty,E);H^\inter) =N\bigpare{(1,\infty); \sqrt{V_0}(H^\inter+V_0-E)^{-1}\sqrt{V_0}}.
\end{equation}
We recall that by our assumptions (see Section~\ref{sec1.2a}), 
we have $V^\inter+V_0\geq E_+$ everywhere, and therefore $H^\inter+V_0\geq E_+$; 
thus, the inverse operator in the r.h.s. of \eqref{eq:counting-eq-lem-2}
is well defined. 
Using tunnelling  and non-trapping estimates, in Section~\ref{sec4} we prove that 
the right hand sides of \eqref{eq:counting-eq-lem-1} and \eqref{eq:counting-eq-lem-2}
coincide for $E$ outside exponentially small neighbourhoods
of resonant energies. This yields \eqref{b5} for relevant values of $\theta$.

\section{Perturbation theory for spectral flow}\label{sec3}

\subsection{Perturbation result}

The proof of Theorem~\ref{main-theorem} is achieved by applying a version of perturbation theory for 
unitary families to the representation \eqref{eq:Smatrix-chain-rule}. Thus, our aim in this section 
is to consider the unitary families of the type $M(t)=\widetilde U(t) U(t)$ and to prove 
the following statement. 

\begin{thm}\label{sf-ptb-thm}
Let $U(t)$ and $\widetilde U(t)$, $t\in [0,1]$, be norm continuous families of unitary operators 
in a Hilbert space $\mathcal{H}$ such that $U(t)-I$ and $\widetilde U(t)-I$ are compact for all $t$. 
Assume that $U(0)=\widetilde U(0)=I$. Assume also that for some $\f\in (0,\pi)$ one has 
\begin{equation}\label{eq:sf-ptb-thm-1}
\s(\widetilde U(1))\subset \bigset{e^{i\chi}}{-\f\leq \chi\leq \f}.
\end{equation}
Denote $M(t)=\widetilde U(t)U(t)$ and 
\begin{equation}\label{eq:sf-ptb-thm-2}
m=\mbox{\rm sf}(-1;\{\widetilde U(t)\}_{t\in[0,1]}). 
\end{equation}
The for any $\th\in (\f,2\pi-\f)$ one has 
\begin{equation}\label{eq:sf-ptb-thm-3}
\begin{split}
\mbox{\rm sf}(e^{i(\th+\f)};\{U(t)\}_{t\in[0,1]})+m 
&\leq \mbox{\rm sf}(e^{i\th};\{M(t)\}_{t\in[0,1]})\\
&\leq \mbox{\rm sf}(e^{i(\th-\f)};\{U(t)\}_{t\in[0,1]})+m.
\end{split}
\end{equation}
\end{thm}

In Section~\ref{sec3b} we prove two particular cases of Theorem~\ref{sf-ptb-thm} and 
in Section~\ref{sec3c} we combine these results to obtain the full proof. 

\subsection{Preliminary statements}\label{sec3b}

The following statement is a version of Theorem~\ref{sf-ptb-thm} with $\f=0$. 

\begin{lem}\label{sf-ptb-lem-1}
Let $U(t)$ and $\widetilde U(t)$, $t\in [0,1]$, be norm continuous families of unitary operators 
such that $U(t)-I$ and $\widetilde U(t)-I$ are compact for all $t$. 
Assume $U(0)=\widetilde U(0)=\widetilde U(1)=I$. Denote $M(t)=\widetilde U(t)U(t)$ and 
$m=\mbox{\rm sf}(-1;\{\widetilde U(t)\}_{t\in[0,1]})$. Then for all $\th\in (0,2\pi)$, 
\begin{equation}\label{eq:sf-ptb-lem-1}
\mbox{\rm sf}(e^{i\th};\{M(t)\}_{t\in[0,1]}) =\mbox{\rm sf}(e^{i\th},\{U(t)\}_{t\in[0,1]})+m.
\end{equation}
\end{lem}

\begin{proof}
We first note that since $\widetilde U(0)=\widetilde U(1)=I$, we have
\begin{equation}\label{eq:sf-is-constant}
\mbox{sf}(e^{i\th};\{\widetilde U(t)\}_{t\in[0,1]})=m, \quad \text{for all }\th\in (0,2\pi).
\end{equation}
Next, in the Hilbert space $\mathcal{H}\oplus\mathcal{H}$, we consider the families 
\[
Q_0(t)=U(t)\oplus \widetilde U(t); \quad Q_1(t) =\widetilde U(t) U(t)\oplus I. 
\]
It is evident that for all $\th\in (0,2\pi)$, 
\begin{align}
& \mbox{sf}(e^{i\th};\{Q_0(t)\}_{t\in[0,1]}) =\mbox{sf}(e^{i\th};\{U(t)\}_{t\in[0,1]})+m, \label{eq:Q0=U+m}\\
& \mbox{sf}(e^{i\th};\{Q_1(t)\}_{t\in[0,1]}) =\mbox{sf}(e^{i\th};\{M(t)\}_{t\in[0,1]}).\label{eq:Q1=M}
\end{align}
We note $Q_0$ and $Q_1$ have the same end points: 
\[
Q_0(0)=Q_1(0)=I; \quad Q_0(1)=Q_1(1)=U(1)\oplus I.
\]
Below we construct a homotopy between $\{Q_0(t)\}_{t\in[0,1]}$ and $\{Q_1(t)\}_{t\in[0,1]}$. 
This will show that the left hand sides of \eqref{eq:Q0=U+m} and \eqref{eq:Q1=M} coincide and therefore 
\eqref{eq:sf-ptb-lem-1} holds true. 

We consider the following norm continuous family of unitary operators on $\mathcal{H}\oplus\mathcal{H}$:
\[
\widetilde U_\t(t) =\begin{pmatrix}\cos(\t\pi/2) & -\sin(\t\pi/2) \\ \sin(\t\pi/2) & \cos(\t\pi/2)\end{pmatrix}
\begin{pmatrix}I & 0 \\ 0 & \widetilde U(t)\end{pmatrix}
\begin{pmatrix}\cos(\t\pi/2) & \sin(\t\pi/2) \\ -\sin(\t\pi/2) & \cos(\t\pi/2)\end{pmatrix},
\]
where $t\in [0,1]$, $\t\in [0,1]$. By inspection, we learn
\begin{enumerate}
\renewcommand{\theenumi}{\roman{enumi}}
\renewcommand{\labelenumi}{(\theenumi) }
\item $\widetilde U_0(t)= I\oplus \widetilde U(t)$ for $t\in [0,1]$; 
\item $\widetilde U_1(t) = \widetilde U(t)\oplus I$ for $t\in [0,1]$;
\item $\widetilde U_\t(0)= \widetilde U_\t(1)= I\oplus I$ for $\t\in[0,1]$;
\item $\widetilde U_\t(t)-I$ is compact for all $t,\t$. 
\end{enumerate}
It follows that the family 
\[
Q_\t(t)=\widetilde U_\t(t)(U(t)\oplus I)\quad \text{in }\mathcal{H}\oplus\mathcal{H}
\]
provides the required homotopy between $Q_0$ and $Q_1$. 
\end{proof}

The following statement is related to the case when $m=0$ in the hypothesis of Theorem~\ref{sf-ptb-thm}. 

\begin{lem}\label{sf-ptb-lem-2}
Let $U$ be a unitary operator in $\mathcal{H}$ such that $U-I$ is compact. Let $A$ be a compact 
self-adjoint operator in $\mathcal{H}$ with $\norm{A}\leq \f<\pi$. Then 
\begin{align}
&\mbox{\rm sf}(e^{i\th};\{e^{itA}U\}_{t\in[0,1]}) \leq N(e^{i(\th-\f)}, e^{i\th};U), \quad \f<\th<2\pi, 
\label{eq:sf-lem-2-1}\\
&\mbox{\rm sf}(e^{i\th};\{e^{itA}U\}_{t\in[0,1]}) \geq -N(e^{i\th},e^{i(\th+\f)};U), \quad 0<\th<2\pi-\f.
\label{eq:sf-lem-2-2}
\end{align}
\end{lem}

\begin{proof}
1) Using the spectral representation of $A$, let us split $A$ into the positive and negative parts: 
\[
A=A_+-A_-, \quad \s(A_\pm)\subset [0,\f].
\]
It is easy to see that $\{e^{itA}\}_{t\in[0,1]}$ is homotopic to the path $\{\widetilde U(t)\}_{t\in[0,2]}$, where 
\[
\widetilde U(t) = 
\begin{cases} 
e^{itA_+}, \quad & 0\leq t\leq 1, \\
e^{-i(t-1)A_-} e^{iA_+}, \quad & 1\leq t \leq 2. 
\end{cases}
\]
Thus we have 
\begin{equation}\label{eq:sf-lem-2-3}
\mbox{sf}(e^{i\th};\{e^{itA}U\}_{t\in[0,1]}) = \mbox{sf}(e^{i\th};\{e^{itA_+}U\}_{t\in[0,1]}) 
+\mbox{sf}(e^{i\th};\{e^{-itA_-}e^{iA_+}U\}_{t\in[0,1]}).
\end{equation}
2) We consider the family $\{e^{itA_+}U\}_{t\in[0,1]}$. Its eigenvalues are branches of analytic functions. 
If $\m_n(t)$ is a simple eigenvalue of $e^{itA_+}U$ with the corresponding normalized egenvector $\g_n(t)$, 
then 
\[
\frac{\m'_n(t)}{i\m_n(t)}=(A_+\g_n(t),\g_n(t))\in [0,\f].
\]
This calculation shows that as $t$ increases, the eigenvalues of $e^{itA_+}U$ rotate anti-clockwise 
with the angular speed of rotation $\leq \f$. From here it clearly follows that 
\begin{equation}\label{eq:sf-lem-2-4}
\mbox{sf}(e^{i\th};\{e^{itA_+}U\}_{t\in[0,1]}) \leq N(e^{i(\th-\f)},e^{i\th};U), 
\quad \f<\th<2\pi. 
\end{equation}
3) A similar argument shows that the eigenvalues of $e^{-itA_-}e^{iA_+}U$ rotate clockwise and therefore
\begin{equation}\label{eq:sf-lem-2-5}
\mbox{sf}(e^{i\th};\{e^{-itA_-}e^{iA_+}U\}_{t\in[0,1]})\leq 0, \quad 0<\th<2\pi.
\end{equation}
Combining \eqref{eq:sf-lem-2-3}--\eqref{eq:sf-lem-2-5}, we obtain the upper bound \eqref{eq:sf-lem-2-1}. 
The lower bound \eqref{eq:sf-lem-2-2} is obtained in a similar way by using the family 
\[
\widetilde U(t) = 
\begin{cases} 
e^{-itA_-}, \quad & 0\leq t\leq 1, \\
e^{i(t-1)A_+} e^{-iA_-}, \quad & 1\leq t \leq 2. 
\end{cases}
\]
\end{proof}

\begin{lem}\label{lemmab3}
Let $\{U(t)\}_{t\in[0,1]}$ be a norm continuous family of unitary 
operators such that $U(t)-I$ is compact for all $t$. 
Then for all $\th_1,\th_2\in(0,2\pi)$ one has
\begin{multline}
\sflow(e^{i\th_1};\{U(t)\}_{t\in[0,1]})
-
\sflow(e^{i\th_2};\{U(t)\}_{t\in[0,1]})
\\
=
N(e^{i\th_1},e^{i\th_2};U(1))
-
N(e^{i\th_1},e^{i\th_2};U(0)).
\label{eqb1}
\end{multline}
\end{lem}
\begin{proof}
First suppose that there exists $\theta_0\in(0,2\pi)$ 
such that \eqref{eq:good-theta} holds true. 
Then, by the definition \eqref{eq:sf-def-1}, the left hand side of 
\eqref{eqb1} becomes
\begin{multline*}
N(e^{i\th_1},e^{i\th_0};U(1))
-
N(e^{i\th_1},e^{i\th_0};U(0))
-
N(e^{i\th_2},e^{i\th_0};U(1))
\\
+
N(e^{i\th_2},e^{i\th_0};U(0))
=
N(e^{i\th_1},e^{i\th_2};U(1))
-
N(e^{i\th_1},e^{i\th_2};U(0)),
\end{multline*}
so the statement is proven. 
The general case follows by applying the above result to each 
of the intervals $\Delta_j$ (see \eqref{eq:sf-def-2}), 
which leads to telescopic sums in both left and right
hand sides of \eqref{eqb1}. 
\end{proof}

\subsection{Proof of Theorem~\ref{sf-ptb-thm}}\label{sec3c}

1) By our assumption \eqref{eq:sf-ptb-thm-1}, we can write $\widetilde U(1)=e^{iA}$, 
where $A$ is a compact self-adjoint operator with $\norm{A}\leq \f$. 
Now set 
\[
\widetilde U(t) =e^{i(2-t)A}, \quad \text{and} \quad M(t)=\widetilde U(t) U(1), \quad\text{for } t\in [1,2].
\]
We obtain 
\begin{align}
\sflow(e^{i\th};\{M(t)\}_{t\in[0,1]}) 
&= \sflow(e^{i\th};\{M(t)\}_{t\in[0,2]}) 
-\sflow(e^{i\th};\{M(t)\}_{t\in[1,2]})
\notag
\\
&=\sflow(e^{i\th};\{M(t)\}_{t\in[0,2]}) 
-\sflow(e^{i\th};\{e^{i(1-t)A}U(1)\}_{t\in[0,1]})
\notag
\\
&=\sflow(e^{i\th};\{M(t)\}_{t\in[0,2]}) 
+\sflow(e^{i\th};\{e^{itA}U(1)\}_{t\in[0,1]}).
\label{eq:sf-ptb-thm-4}
\end{align}
2) Since $-1\notin \s(e^{itA})=I$ for $t\in[0,1]$, it is easy to see that
\[
\mbox{sf}(-1;\{\widetilde U(t)\}_{t\in[0,2]}) = \mbox{sf}(-1,\{\widetilde U(t)\}_{t\in[0,1]})=m.
\]
Since $\widetilde U(0)=\widetilde U(2)$, we can apply 
Lemma~\ref{sf-ptb-lem-1} to the family $\{M(t)\}_{t\in[0,2]}$. 
This yields
\begin{equation}\label{eq:sf-ptb-thm-5}
\mbox{sf}(e^{i\th}; \{M(t)\}_{t\in[0,2]}) =\mbox{sf}(e^{i\th};\{U(t)\}_{t\in[0,1]})+m.
\end{equation}
Combining \eqref{eq:sf-ptb-thm-4}, \eqref{eq:sf-ptb-thm-5} and 
the estimates of Lemma~\ref{sf-ptb-lem-2}, we obtain 
\begin{align}
\mbox{sf}(e^{i\th};\{M(t)\}_{t\in[0,1]})&\leq m+ \mbox{sf}(e^{i\th};\{U(t)\}_{t\in[0,1]})\nonumber \\
&\quad +N(e^{i(\th-\f)},e^{i\th};U(1)), \quad \f<\th<2\pi, \label{eq:sf-ptb-thm-6} \\
\mbox{sf}(e^{i\th};\{M(t)\}_{t\in[0,1]})&\geq m+ \mbox{sf}(e^{i\th};\{U(t)\}_{t\in[0,1]})\nonumber \\
&\quad -N(e^{i\th},e^{i(\th+\f)};U(1)), \quad 0<\th<2\pi-\f.\label{eq:sf-ptb-thm-7}
\end{align}
3) 
By Lemma~\ref{lemmab3}, taking into account $U(0)=I$, we get
\begin{equation}
\label{eq:sf-ptb-thm-8}
\sflow(e^{i(\th-\f)}; \{U(t)\}_{t\in[0,1]}) 
=
\sflow(e^{i\th};\{U(t)\}_{t\in[0,1]})
+N(e^{i(\th-\f)},e^{i\th};U(1))
\end{equation}
for all $\f<\th<2\pi$.
Combining \eqref{eq:sf-ptb-thm-6} and \eqref{eq:sf-ptb-thm-8} yields 
the upper bound in \eqref{eq:sf-ptb-thm-3}. The lower bound is obtained in the same way from 
\eqref{eq:sf-ptb-thm-7}.  \qed


\section{Proof of Theorem~\ref{main-theorem}}\label{sec4}

\subsection{Non-trapping and tunnelling resolvent estimates} 
The analytic basis of our proof is provided by Propositions~\ref{prop5.1} and \ref{prop5.2} 
below. 
The first of these results yields a semiclassical resolvent estimate
for non-trapping potentials:
\begin{prop}\cite{RT1,RT2,GM}\label{prop5.1}
Suppose that a potential $\widetilde V$ satisfies Assumption~A with some
$\rho>0$ and let $E>0$ be a non-trapping energy for $\widetilde V$ 
(i.e. Assumption~B(ii) holds true with $\re^d$ instead of $\mathcal{G}^\exter(E)$).
Then for any $s>1/2$ we have the estimate
$$
\norm{\jap{x}^{-s}(H_0+\widetilde V-E-i0)^{-1}\jap{x}^{-s}}\leq O(\h^{-1})
$$
as $\h\to0$. Furthermore, if $E$ ranges over a compact interval in 
a non-trapping energy range, then the above bound is uniform in $E$. 
\end{prop}

The second result crucial for us is known as a tunnelling estimate; 
it goes back to Agmon \cite{Ag},
see also \cite{HS2,Na0} or \cite[Section~6]{DS}.
Fix a compact set $K\subset \re^d$ and let $\chi_K$ be the characteristic function 
of $K$ in $\re^d$.
Consider a potential $\widetilde V\in C(\re^d)$, $\widetilde V\geq 0$, 
and let $E<0$. Let 
$\mathbf{d}x^2 = (\widetilde V(x)-E)_+dx^2$ be the Agmon metric for $\widetilde V$ at the energy 
$E$ and let $\mathbf{d}(x,K)$ be the corresponding Agmon distance from $x$ to $K$. 

\begin{prop}\cite[Section~6]{DS}\label{prop5.2}
For any $\varepsilon>0$ there exists $C_\varepsilon>0$ such that for all sufficiently small $\h>0$, the estimate
$$
\norm{e^{(\mathbf{d}(x,K)-\varepsilon)/\h}(H_0+\widetilde V-E)^{-1}\chi_K}\leq C_\varepsilon
$$
holds true. 
\end{prop}

Combining the above two results, we obtain the  following key estimates:

\begin{lem}\label{tunnel-est-lem-1}
There exist positive constants $\d$, $\alpha$ such that for all $E$ satisfying $|E-E_0|<\d$ one has 
\begin{align}
\label{eq:tunnel-est-1}
&\norm{B^\exter(E)}
\leq e^{-2\alpha/\h},
\\
\label{eq:tunnel-est-2}
&\bignorm{A^\exter(E)-\sqrt{V_0}(H^\inter+V_0-E)^{-1}\sqrt{V_0}}
\leq e^{-2\alpha/\h},
\end{align}
provided $\h>0$ is sufficiently small. 
\end{lem}

\begin{proof}
1)
We choose $\d>0$ sufficiently small that $V^\inter(x)+V_0(x)\geq E_0+2\d$ for all $x\in\re^d$. 
Denote
\[
D(E)
= 
\sqrt{V_0}\bigpare{R^\exter(E+i0)-(H^\inter+V_0-E)^{-1}}\sqrt{V_0}.
\]
Since $(H^\inter+V_0-E)^{-1}$ is self-adjoint, recalling the definition of operators $A^\exter$, $B^\exter$, 
we obtain
$$
A^\exter(E)-\sqrt{V_0}(H^\inter+V_0-E)^{-1}\sqrt{V_0}
=
\Re D(E), 
\quad 
B^\exter(E)=\Im D(E).
$$
Thus, it suffices to prove the estimate
\begin{equation}
\norm{D(E)}\leq e^{-\k/\h}
\label{eq5.3} 
\end{equation}
for $\abs{E-E_0}<\delta$ and $\h>0$ sufficiently small. 

2) Let us prove \eqref{eq5.3}. 
By the second resolvent equation, we have
\begin{align}
D(E)&=-\sqrt{V_0} R^\exter(E+i0)(V-V^\inter)(H^\inter+V_0-E)^{-1}\sqrt{V_0}
\notag
\\
&= -\sqrt{V_0} R^\exter(E+i0)\jap{x}^{-s} \cdot   \jap{x}^s(V-V^\inter)(H^\inter+V_0-E)^{-1}\sqrt{V_0},
\label{eq5.4}
\end{align}
where $s>1/2$. 
By Proposition~\ref{prop5.1}, we have 
\begin{equation}
\norm{\sqrt{V_0}R^\exter(E+i0)\jap{x}^{-s}}
\leq 
C\h^{-1}
\label{eq5.5}
\end{equation}
for small $\h>0$. 
Next, let 
$\mathbf{d}x^2 = (V^\inter+V_0-E)dx^2$ be the Agmon metric
for the potential $V^\inter+V_0$, 
and let $\mathbf{d}_E(x,K)$ 
be the corresponding Agmon distance from 
$x\in\re^d$ to the set $K=\supp V_0$. 
By Proposition~\ref{prop5.2}, we have
\begin{equation}
\norm{e^{-(\mathbf{d}_E(x,K)-\varepsilon)/\h}(H^\inter+V_0-E)^{-1}\chi_K}\leq C_\varepsilon.
\label{eq5.6}
\end{equation}
Now note that 
$$
\supp V_0\cap \supp(V-V^\inter)\subset \Omega_1\cap \Omega_2^c = \varnothing. 
$$
Thus, \eqref{eq5.6} yields
\begin{equation}
\norm{\jap{x}^s(V-V^\inter)(H^\inter+V_0-E)^{-1}\sqrt{V_0}}\leq C_\e e^{-(\mathbf{d}-\e)/\h}
\label{eq5.7}
\end{equation}
with any $\e>0$, where
$$
\mathbf{d}=\inf\{\mathbf{d}_E(x,K)\mid x\in\supp(V-V^\inter),\quad \abs{E-E_0}<\delta\}>0.
$$
Combining \eqref{eq5.4}, \eqref{eq5.5} and \eqref{eq5.7}, we obtain \eqref{eq5.3}
with any $\alpha<\mathbf{d}/2$. 
\end{proof}

\subsection{Relating $\mu(\cdot,E;H,H^\exter)$ to $H^\inter$} 

\begin{lem}\label{counting-eq-lem}
Let the constants $\alpha$, $\delta$ be as in Lemma~\ref{tunnel-est-lem-1}. 
Then for all $E$ satisfying $|E-E_0|<\d$ and 
$\dist(E,\s(H^\inter))>e^{-\alpha/\h}$, one has 
\begin{equation}\label{eq:counting-eq}
\m(-1,E;H,H^\exter)=N((-\infty,E);H^\inter),
\end{equation}
provided $\h>0$ is sufficiently small. 
\end{lem}

\begin{proof}
1) 
Assume $\dist(E,\s(H^\inter))>e^{-\alpha/\h}$; 
our aim is to show that the right hand sides of \eqref{eq:counting-eq-lem-1} and \eqref{eq:counting-eq-lem-2} 
coincide for small $\h$. 

\noindent
2) We use the operator identity 
\begin{equation}\label{eq:counting-eq-lem-3}
\bigpare{I-\sqrt{V_0}(H^\inter+V_0-E)^{-1}\sqrt{V_0}} \bigpare{I+\sqrt{V_0}(H^\inter-E)^{-1}\sqrt{V_0}}=I.
\end{equation}
Using our assumption $\dist(E,\s(H^\inter))>e^{-\alpha/\h}$, we obtain 
\begin{align*}
\bignorm{I+\sqrt{V_0}(H^\inter-E)^{-1}\sqrt{V_0}}
&\leq 1+ \bignorm{\sqrt{V_0}}\cdot \norm{(H^\inter-E)^{-1}}\cdot \bignorm{\sqrt{V_0}}\\
&\leq 1+ e^{\alpha/\h}\norm{V_0}.
\end{align*}
Therefore, by \eqref{eq:counting-eq-lem-3}, 
\[
\Bignorm{\Bigpare{I-\sqrt{V_0}(H^\inter+V_0-E)^{-1} \sqrt{V_0}}^{-1}}\leq 1+ e^{\alpha/\h}\norm{V_0}.
\]
It follows that 
\begin{equation}\label{eq:counting-eq-lem-4}
\dist\bigpare{1,\s\bigpare{\sqrt{V_0}(H^\inter+V_0-E)^{-1}\sqrt{V_0}}}
\geq 
\frac{1}{1+e^{\alpha/\h}\norm{V_0}} 
>\frac{e^{-\alpha/\h}}{1+\norm{V_0}}.
\end{equation}
3) By the estimates \eqref{eq:tunnel-est-2} and \eqref{eq:counting-eq-lem-4}, for all 
sufficiently small $\h$ we have 
\begin{align*}
&\bignorm{A^\exter(E)-\sqrt{V_0}(H^\inter+V_0-E)^{-1}\sqrt{V_0}}
\leq 
e^{-2\alpha/\h}<\frac{e^{-\alpha/\h}}{1+\norm{V_0}}
\\
&\qquad 
<
\dist\bigpare{1,\s\bigpare{\sqrt{V_0}(H^\inter+V_0-E)^{-1}\sqrt{V_0}}},
\end{align*}
and so, applying the elementary perturbation theory for compact self-adjoint operators,
we get that
the right hand sides of \eqref{eq:counting-eq-lem-1} and \eqref{eq:counting-eq-lem-2} coincide. 
\end{proof}

\subsection{An estimate for $S(E;H,H^\exter)-I$}

We will need a corollary of 
Proposition~\ref{BS-count-prop}: 

\begin{lem}\label{S-mat-bound-lem}
Let $a=\dist(1,\s(A^\exter(E)))>0$. Then 
\begin{equation}\label{eq:S-mat-bound}
\norm{S(E;H,H^\exter)-I}\leq 
2a^{-1} \norm{B^\exter(E)}\bigpare{1+a^{-2}\norm{B^\exter(E)}^2}^{-1/2}.
\end{equation}
\end{lem}

\begin{proof}
By the identity \eqref{eq:S-mat-kernel} and elementary perturbation theory for self-adjoint operators, 
\[
\text{if }|\cot (\th/2)|\cdot\norm{B^\exter(E)}< a,\quad 
\text{then } e^{i\th}\notin \s(S(E;H,H^\exter)).
\]
Let $\th_0\in (0,\pi)$ be such that 
$\cot (\th_0/2)\norm{B^\exter(E)}=a$. 
Then 
\[
\s(S(E;H,H^\exter))\subset \bigset{e^{i\th}}{-\th_0<\th<\th_0},
\]
and so 
\[
\norm{S(E;H,H^\exter)-I}\leq |e^{i\th_0}-1|. 
\]
Now an elementary calculation shows that 
\[
|e^{i\th_0}-1| =2 \bigpare{1+a^2\norm{B^\exter(E)}^{-2}}^{-1/2},
\]
which proves \eqref{eq:S-mat-bound}. 
\end{proof}

\begin{lem}\label{non-reso-S-mat-lem}
Let $\delta$, $\alpha$ be the constants from Lemma~\ref{tunnel-est-lem-1}.
Then for all $E$ satisfying $|E-E_0|<\d$ and 
$\dist(E,\s(H^\inter))>e^{-\alpha/\h}$, one has 
\begin{equation}\label{eq:non-reso-S-mat}
\norm{S(E;H,H^\exter)-I}\leq C e^{-\alpha/\h}
\end{equation}
for all sufficiently small $\h>0$. 
\end{lem}

\begin{proof}
Let $\dist(E,\s(H^\inter))>e^{-\alpha/\h}$. 
Then, as in the proof of Lemma~\ref{counting-eq-lem}, 
(see \eqref{eq:counting-eq-lem-4}), we get
$$
\dist\bigpare{1,\s\bigpare{\sqrt{V_0}(H^\inter+V_0-E)^{-1}\sqrt{V_0}}}
>
c\, e^{-\alpha/\h}.
$$
Combining this with the estimate \eqref{eq:tunnel-est-2}
of Lemma~\ref{tunnel-est-lem-1}, we obtain, for sufficiently small $\h$: 
\[
\dist(1,\s(A^\exter(E))) 
\geq 
c\, e^{-\alpha/\h}-e^{-2\alpha/\h}\geq (c/2) e^{-\alpha/\h}.
\]
Combining this with the estimate \eqref{eq:tunnel-est-1} of
Lemma~\ref{tunnel-est-lem-1} and with Lemma~\ref{S-mat-bound-lem}, 
we obtain:
\begin{align*}
\norm{S(E;H,H^\exter)-I} 
&
\leq 
2\norm{B^\exter(E)}\cdot (\dist(1,\s(A^\exter(E))))^{-1} 
\\
&\leq 2 e^{-2\alpha/\h}2c^{-1} e^{\alpha/\h} \leq (4/c)e^{-\alpha/\h},
\end{align*}
as required.
\end{proof}

\subsection{Proof of Theorem~\ref{main-theorem}}
We use Theorem~\ref{sf-ptb-thm} with 
\begin{align*}
&U(t) =S(E';H^\exter,H_0), \\
&\widetilde U(t) =\widetilde S(E'; H,H^\exter), \quad E'=E+\frac{1-t}{t}, \quad t\in (0,1],
\end{align*}
and $U(0)=\widetilde U(0)=I$.  Here $\widetilde S(E;H,H^\exter)$ is as defined in Section~3.1, 
see \eqref{stilde2}.
Then, according to the chain rule 
\eqref{eq:Smatrix-chain-rule}, we have 
\[
M(t)= \widetilde U(t) U(t) = S(E;H,H_0).
\]
By Lemma~\ref{non-reso-S-mat-lem}, the hypothesis \eqref{eq:sf-ptb-thm-1} of Theorem~\ref{sf-ptb-thm} is 
satisfied with $\f=O(e^{-\a/\h})$ as $\h\to 0$. By Lemma~\ref{counting-eq-lem}, 
\[
m=N((-\infty,E);H^\inter).
\]
Now the conclusion of Theorem~\ref{sf-ptb-thm} yields the desired estimates. 
\qed

\appendix

\section{Proof of Proposition~\ref{S-mat-ev-bound}} \label{S-mat-ev-bound-proof}

1) Without loss of generality we assume $\th_2=2\pi-\th_1$. Then is suffices to prove that 
\[
N((-\infty,\cos\th_1);\Re S(E;H^\exter,H_0)) 
=O(\h^{-\eta}), \quad \h\to 0. 
\]
We denote the $q$-th Schatten trace ideal class by $S_q$. Due to the estimate:
\[
N((-\infty,-a);A) \leq a^{-q}\norm{A}_{S_q}^q \quad\text{for all } a>0, \ q\geq 1, 
\]
it suffices to prove that 
\begin{equation}\label{eq:trace-norm-est}
\norm{\Re S(E;H^\exter,H_0)-I}_{S_q}^q = O(\h^{-b}), \quad \h\to 0,
\end{equation}
with some exponents $q\geq 1$ and $b>0$. 

\noindent 
2) According to the stationary representation \eqref{eq:S-mat-rep-1} for the scattering matrix, we have 
\begin{equation}\label{eq:Re-S-mat-rep}
\Re S(E;H^\exter,H_0)-I =
2\pi\  \Im\! \Bigpare{ \mathcal{F}_E V^\exter R^\exter(E+i0) V^\exter \mathcal{F}_E^*}. 
\end{equation}
In order to estimate the norm of the operator in the r.h.s.\ of \eqref{eq:Re-S-mat-rep}, we use the 
non-trapping resolvent estimate of Proposition~\ref{prop5.1} and also 
the following Schatten class estimate: 
\begin{equation}\label{eq:trace-lemma}
\norm{\mathcal{F}_E\jap{x}^{-s}}_{S_q}\leq C \h^{-r},
\end{equation}
if $\frac12<r<s$ and $q=2(d-1)/(2r-1)$. \eqref{eq:trace-lemma} follows from an interpolation between 
\begin{equation}\label{eq:trace-lemma-0} 
\norm{\mathcal{F}_E\jap{x}^{-s_0}}\leq C\h^{-1/2}\quad\text{if }s_0>1/2
\end{equation}
and 
\[
\norm{\mathcal{F}_E\jap{x}^{-s_1}}_{S_2}\leq C\h^{-d/2}\quad \text{if } s_1>d/2.
\]
Proposition~\ref{prop5.1} and the estimate  \eqref{eq:trace-lemma}
 imply 
\[
\norm{\mathcal{F}_E V^\exter R^\exter(E+i0) V^\exter \mathcal{F}_E^*}_{S_q}\leq C\h^{-1-2r}
\]
with $\frac12<r<s=\rho/2$, and $q=(d-1)/(2r-1)$, where $\rho>1$ is the constant in the assumption
\eqref{eq:Short-Range-1}. This proves \eqref{eq:trace-norm-est} with $b=1+2r$. \qed 

\section{Proof of Proposition~\ref{Smatrix-Hoelder}}\label{Smatrix-Hoelder-proof}

Let $s>\frac12$ and $\c=(s-\frac12)(s+\frac12)$. Then we have 
\begin{equation}\label{eq:App-B-1}
\norm{\jap{x}^{-s} R^\exter(E_1)\jap{x}^{-s} -\jap{x}^{-s} R^\exter(E_2)\jap{x}^{-s}}
\leq C \h^{-1-\c} |E_1-E_2|^\c
\end{equation}
if $E_1,E_2\in (E_0-\d,E_0+\d)$ with small $\d>0$. 
\eqref{eq:App-B-1} is shown by observing the $\h$-dependence of the constants
in the proofs of the key statements of the Mourre theory. 
See, e.g., \cite{HN}, Section~2. There the $\h$-dependence of the H\"older 
continuity is not investigated, but we easily observe the estimate by
scaling.

Similarly, we have 
\[
\norm{\mathcal{F}_{E_1}\jap{x}^{-s}-\mathcal{F}_{E_2}\jap{x}^{-s}}^2 \leq C \h^{-1-\c}|E_1-E_2|^\c. 
\]
Now combining these estimates,  \eqref{eq:trace-lemma-0}, and Proposition~\ref{prop5.1}
with the stationary representation formula of the scattering matrix \eqref{eq:S-mat-rep-1}, 
we conclude the proof of the  assertion. 
\qed

\section{Proof of Proposition~\ref{prop.res-error}}\label{app.c}

We use an argument which is essentially due to \cite{CDKS}.

\begin{lem}\label{lma.appc}
Let $A_1$, $A_2$ be self-adjoint operators and let $\lambda\in\re$ be 
such that $\e=\dist(\lambda,\sigma(A_1))>0$. 
Suppose that 
\begin{equation}
\norm{(A_1-\lambda-i\e)^{-1}-(A_2-\lambda-i\e)^{-1}}< \frac1{2\e}. 
\label{appc1}
\end{equation}
Then $\lambda\notin\sigma(A_2)$.
\end{lem}
\begin{proof}
For simplicity of notation, assume $\lambda=0$. 
Set 
$$
N_j(z)=((A_j-i\e)^{-1}-(z-i\e)^{-1})^{-1}, 
\quad 
j=1,2, \quad z\in\co.
$$
By a direct calculation, 
\begin{equation}
N_j(z)=-(z-i\e)-(z-i\e)^2(A_j-z)^{-1}.
\label{appc2}
\end{equation}
Using the last formula and our assumptions $\e=\dist(\lambda,\sigma(A_1))$, we get
\begin{equation}
\norm{N_1(0)}\leq \e+\e^2\norm{A_1^{-1}}\leq 2\e.
\label{appc3}
\end{equation}
By \eqref{appc2}, it suffices to check the boundedness of the norm
of $N_2(z)$ as $z\to0$. Using the resolvent identity, we get
$$
N_2(z)(I-DN_1(z))=N_1(z), 
$$
where $D=(A_1-i\e)^{-1}-(A_2-i\e)^{-1}$. 
By \eqref{appc1} and \eqref{appc3} we get $\norm{DN_1(0)}<1$ 
and so the norm of $N_2(z)$ is bounded when $z\to0$. 
\end{proof}

\begin{proof}[Proof of Proposition~\ref{prop.res-error}]
1)
For $j=1,2$, let $\wV_j$ be the smooth function given by
$$
\wV_j(x)=
\begin{cases}
V^\exter(x) & \text{ if $x\in\Omega_2$}, 
\\
V^\inter_j(x) & \text{ if $x\in\re^d\setminus\Omega_2$}.
\end{cases}
$$
We note that $\wV_j\geq E_+'>E_+$ everywhere and therefore
\begin{equation}
\norm{(H_0+\wV_j-z)^{-1}}\leq C\quad \text{ uniformly in $\Re z\leq E_+$. }
\label{appc4}
\end{equation}
Next, let $W=V_2^\inter-V_1^\inter$, which is supported inside 
$\mathcal{G}^\exter(E_+)$, and let $\chi_W$ be the characteristic function 
of $\supp W$. 
We have 
$$
(H_0+V_j^\inter-z)^{-1}-(H_0+\wV_j-z)^{-1}
=
(H_0+V_j^\inter-z)^{-1}
(\wV_j-V^\inter_j)
(H_0+\wV_j-z)^{-1}, 
$$
and therefore
\begin{multline*}
\norm{(H_0+V_j^\inter-z)^{-1}\chi_W}
\\
\leq
\norm{(H_0+V_j^\inter-z)^{-1}\chi_W}
+
\norm{(H_0+V_j^\inter-z)^{-1}}
\norm{(\wV_j-V^\inter_j)
(H_0+\wV_j-z)^{-1}\chi_W}.
\end{multline*}
Note that the supports of $W$ and $\wV_j-V_j^\inter$ are disjoint. 
Thus, if $\Re z\leq E_+$, then by a tunnelling estimate
(a version of Proposition~\ref{prop5.2} with complex $E$, 
see \cite[Theorem~2.5]{Na0})
there exists $\nu>0$ such that 
$$
\norm{(\wV_j-V_j^\inter)(H_0-\wV_j-z)^{-1}\chi_W}
\leq 
Ce^{-\nu/\h}.
$$
Using \eqref{appc4}, we obtain that
\begin{equation}
\norm{(H_0+V_j^\inter-z)^{-1}\chi_W}\leq C'
\label{appc5}
\end{equation}
for some $C'>0$, provided that $\Re z\leq E_+$ and 
$\norm{(H_0+V_j^\inter-z)^{-1}}\leq e^{-\nu/\h}$. 

2) 
Now we use a \emph{zooming argument} of \cite[Section IV]{CDKS}
to compare the eigenvalues of $H_0+V_1^\inter$ and $H_0+V_2^\inter$. 
Let $\lambda<E_+$ be such that $\dist(\lambda, \sigma(H_0+V_1^\inter))>e^{-\nu/\h}$,
and set $\e=e^{-\nu/\h}$. 
By \eqref{appc5}, we get
\begin{multline*}
\norm{
(H_0+V_1^\inter-\lambda-i\e)^{-1}
-
(H_0+V_2^\inter-\lambda-i\e)^{-1}
}
\\
=
\norm{(H_0+V_1^\inter-\lambda-i\e)^{-1}W(H_0+V_2^\inter-\lambda-i\e)^{-1}}
\leq C.
\end{multline*}
Now using Lemma~\ref{lma.appc} with $A_j=H_0+V_j^\inter$ and 
$\e=e^{-\nu/\h}$, we get that $\lambda\notin\sigma(H_2+V_2^\inter)$.
Finally, using a standard argument involving continuous deformation of
$V_1^\inter$ into $V_2^\inter$, we get the required statement. 
\end{proof}



\begin{thebibliography}{99}
\bibitem{Ag}
Agmon, S.:
\emph{Lectures on exponential decay of solutions of second-order elliptic equations,}
Mathematical Notes, 29. Princeton University Press, Princeton, NJ; 
University of Tokyo Press, Tokyo, 1982.
\bibitem{APS} Atiyah, M. F., Patodi, V. K., Singer, I. M.:
\emph{Spectral asymmetry and Riemannian geometry. III.}
Math. Proc. Cambridge Philos. Soc. {\bf 79} (1976), 71--99.
\bibitem{BY} Birman, M. Sh., Yafaev, D. R.:
\emph{The spectral shift function. The papers of M. G. Kre\v{\i}n and their further development.}
St. Petersburg Math. J. {\bf 4} (1993), 833--870.
\bibitem{CDKS} Combes, J.-M., Duclos, P., Klein, M., Seiler, R.:
\emph{The shape resonance.} 
Comm. Math. Phys. {\bf 110} (1987), 215--236.
\bibitem{DS}
Dimassi, M., Sj\"ostrand, J.:
\emph{Spectral asymptotics in the semi-classical limit.}
LMS Lecture notes series, {\bf 268}. Cambridge University Press, 1999.
\bibitem{GM}
G\'erard, C., Martinez, A.,
\emph{Principe d'absorption limite pour des op\'erateurs de Schršdinger \`a longue port\'ee.}
Comptes rendus de l'Acad\'emie des sciences, {\bf 306} (1988), 121--123.
\bibitem{GMR}
G\'erard, C., Martinez, A., Robert, D.:
\emph{Breit-Wigner formulas for the scattering phase and the 
total scattering cross-section in the semi-classical limit.}
Comm. Math. Phys. {\bf 121} (1989), no. 2, 323--336. 
\bibitem{GS}
G\'erard, C., Sigal, I.M.: 
\emph{Space-time picture of semiclassical resonances.}
Comm. Math. Phys. {\bf 145} (1992), 281--328.
\bibitem{HS} Helffer, B., Sj\"ostrand, J.:
\emph{R\'esonances en limite semi-classique.}
M\'em. Soc. Math. France (N.S.) {\bf 24--25} (1986), iv+228 pp.
\bibitem{HS2} Helffer, B., Sj\"ostrand, J.:
\emph{Multiple wells in the semi-classical limit I.}
Commun. P.~D.~E. {\bf 9} (1984) 337--408. 
\bibitem{HN} Hislop, P. D., Nakamura, S.:
\emph{Semiclassical resolvent estimates.}
Ann. Inst. H. Poincar\'e Phys. Th\'eor. {\bf 51} (1989), 187--198.
\bibitem{HiS}
Hislop, P. D.; Sigal, I. M.:
\emph{Semiclassical theory of shape resonances in quantum mechanics.}
Mem. Amer. Math. Soc. {\bf 78} (1989), no. 399.
\bibitem{Kato} 
Kato, T.:
\emph{Monotonicity theorems in scattering theory.}
Hadronic J. {\bf 1} (1978), no. 1, 134--154. 
\bibitem{Landau-Lifshits}
Landau, L. D. and Lifshitz, E. M.:
\emph{Quantum Mechanics (Non-Relativistic Theory) Course of Theoretical Physics , Volume 3.}
Pergamon Press, 1958. 
\bibitem{MRS}
Martinez, A., Ramond, T., Sj\"ostrand, J.:
\emph{Resonances for nonanalytic potentials.} 
Anal. PDE {\bf 2} (2009), no. 1, 29--60. 
\bibitem{Na2}
Nakamura, S.: 
\emph{Scattering theory for the shape resonance model. I. Nonresonant energies.
II. Resonance scattering.}
Ann. Inst. H. Poincare Phys. Theor. {\bf 50} (1989), no. 2, 115--131, 133--142.
\bibitem{Na3}
Nakamura, S.: 
\emph{Shape resonances for distortion analytic Schršdinger operators,}
Commun. P. D. E. {\bf 14} (1989), 1385--1419.
\bibitem{Na0} Nakamura, S.: 
\emph{Agmon-type exponential decay estimates for pseudodifferential operators.}
J. Math. Sci. Univ. Tokyo {\bf 5} (1998), 693--712. 
\bibitem{Na1} Nakamura, S.: 
\emph{Spectral shift function for trapping energies in the semiclassical limit.}
Commun. Math. Phys. {\bf 208} (1999),173--193. 
\bibitem{Newton}
 Newton, R. G.: 
 \emph{Scattering theory of waves and particles.} 
 McGraw-Hill Book Co., New York-Toronto, Ont.-London 1966
\bibitem{Push1} 
Pushnitski, A.: 
\emph{The spectral shift function and the invariance principle.}
J. Funct. Anal. {\bf 183} no.2 (2001), 269--320.
\bibitem{Push2} 
Pushnitski, A.: 
\emph{The Birman-Schwinger principle on the essential spectrum.}
J. Funct. Anal. {\bf 261} (2011), 2053--2081.
\bibitem{RT1} 
Robert, D., Tamura, H.:  
\emph{Semiclassical bounds for resolvents of Schr\"odinger operators and asymptotics for scattering phases.} 
Comm. Partial Differential Equations {\bf 9} (1984), 1017--1058.
\bibitem{RT2} 
Robert, D., Tamura, H.:  
\emph{Semiclassical estimates for resolvents and asymptotics for total scattering cross-sections.}
Ann. Inst. H. Poincar\'e Phys. Th\'eor. {\bf 46} (1987), 415--442.
\bibitem{RoSa} 
Robbin, J., Salamon, D.:  
\emph{The spectral flow and the Maslov index.}
Bull. London Math. Soc. {\bf 27} (1995),  1--33.
\bibitem{SY} 
Sobolev, A. V., Yafaev, D. R.: 
\emph{On the quasiclassical limit of the total scattering cross section in nonrelativistic quantum mechanics.}
Ann. Inst. H. Poincar\'e Phys. Th\'eor. {\bf 44} (1986), 195--210.
\bibitem{Y} 
Yafaev, D. R.:
\emph{Resonance scattering on a negative potential.}
Journal of Mathematical Sciences {\bf 32} no. 5 (1986), 549--556.
\bibitem{Yafaev} 
Yafaev, D. R.:
\emph{Mathematical scattering theory. Analytic theory.}
A.M.S. 2009.


\end{thebibliography}
\end{document}